\documentclass[11pt]{amsart}
\usepackage{amsmath}
\usepackage[psamsfonts]{amssymb}
\usepackage{graphicx}
\usepackage[english]{babel}
\usepackage[all]{xy}
\usepackage{hyperref}
\usepackage{verbatim}

\theoremstyle{definition}
\newtheorem{theorem}{Theorem}[section]
\newtheorem{prop}[theorem]{Proposition}
\newtheorem{lemma}[theorem]{Lemma}
\newtheorem{cor}[theorem]{Corollary}

\newtheorem{question}[theorem]{Question}

\newtheorem{ex}[theorem]{Example}
\theoremstyle{remark}

\newtheorem{remark}[theorem]{Remark}

\def\co{\colon\thinspace}

\def\R{\mathbb{R}}

\DeclareMathOperator{\Cal}{Cal}

\title[$C^0$ convergence and the Calabi homomorphism]{Graphicality, $C^0$ convergence, and the Calabi homomorphism}
\author{Michael Usher}

\begin{document}
\begin{abstract}
Consider a sequence of compactly supported Hamiltonian diffeomorphisms $\phi_k$ of an exact symplectic manifold, all of which are ``graphical'' in the sense that their graphs are identified by a Darboux-Weinstein chart with the image of a one-form.  We show by an elementary argument that if the $\phi_k$ $C^0$-converge to the identity then their Calabi invariants converge to zero.  This generalizes a result of Oh, in which the ambient manifold was the two-disk and an additional assumption was made on the Hamiltonians generating the $\phi_k$.  We discuss connections to the open problem of whether the Calabi homomorphism extends to the Hamiltonian homeomorphism group.  The proof is based on a relationship between the Calabi invariant of a $C^0$-small Hamiltonian diffeomorphism and the generalized phase function of its graph.
\end{abstract}

\maketitle

\section{Introduction}
Let $(M,\lambda)$ be an exact symplectic manifold, that is, $M$ is a smooth, necessarily noncompact manifold of some even dimension $2n$ endowed with a $1$-form $\lambda$ such that $d\lambda$ is nondegenerate.  A compactly supported smooth function $H\co [0,1]\times M\to\R$ defines a time-dependent vector field $\{X_{H_t}\}_{t\in[0,1]}$ by the prescription that $d\lambda(X_{H_t},\cdot)=d(H(t,\cdot))$; the (compactly-supported) Hamiltonian diffeomorphism group of the symplectic manifold $(M,d\lambda)$, which we will write as $Ham^c(M,d\lambda)$, is by definition the group consisting of the time-one maps of such Hamiltonian vector fields. We will denote the time-$t$ map of the Hamiltonian vector field associated to $H\co [0,1]\times M\to \R$ as $\phi_{H}^{t}$.

In \cite{Cal}, Calabi introduced a homomorphism $\Cal\co Ham^c(M,d\lambda)\to\R$ which is determined by the formula \begin{equation}\label{calH} \Cal(\phi_{H}^{1})=\int_{[0,1]\times M} Hdt\wedge (d\lambda)^{n}.\end{equation}
The fundamental importance of this homomorphism is reflected in the fact that, as follows easily from the simplicity of  $\ker(\Cal)$ \cite[Th\'eor\`eme II.6.2]{B} and the obvious fact that $Ham^c(M,d\lambda)$ has trivial center, any non-injective homomorphism $f\co Ham^c(M,d\lambda)\to G$ to any group $G$ must factor through $\Cal$ as $f=g\circ \Cal$ for some homomorphism $g\co \R\to G$.  

At least since the proof of the Eliashberg-Gromov rigidity theorem \cite{E} it has been understood that many features of symplectic geometry survive under appropriate $C^0$ limits; it is natural to ask whether this principle applies to the Calabi homomorphism.  Specifically, we may consider the following question:

\begin{question}\label{mainq} Let $(M,\lambda)$ be an exact symplectic manifold and let $\{H_k\}_{k=1}^{\infty}$ be a sequence of smooth functions all having support in some compact subset $K\subset M$, such that $\xymatrix{\phi_{H_k}^{1}\ar[r]^{C^0}&1_M}$ and such that there is $H\in L^{(1,\infty)}([0,1]\times M)$ such that $\xymatrix{H_k\ar[r]^{L^{(1,\infty)}}&H}$.  Must it be true that $\Cal(\phi_{H_k}^{1})\to 0$?  (In other words, must it be true that $\int_{[0,1]\times M}Hdt\wedge(d\lambda)^n=0$?)
\end{question}

 Here the $L^{(1,\infty)}$ norm on the space of compactly supported smooth functions $C^{\infty}_{c}([0,1]\times M)$ is defined by $\|H\|_{L^{(1,\infty)}}=\int_{0}^{1}\max|H(t,\cdot)|dt$ and\\ $L^{(1,\infty)}([0,1]\times M)$ is the completion of $C^{\infty}_{c}([0,1]\times M)$ with respect to this norm.  It would also be reasonable to replace the $L^{(1,\infty)}$ norm in Question \ref{mainq} by the $C^0$ norm; our choice of the $L^{(1,\infty)}$ norm is mainly for consistency with \cite{Oh16},\cite{OM}.  However it should be emphasized that \emph{some} control over the Hamiltonian functions $H_k$ and not just on the time-one maps $\phi_{H_k}^{1}$ is necessary in Question \ref{mainq}: the reader should not find it difficult to construct sequences of Hamiltonian diffeomorphisms $\phi_{H_k}^{1}$ all having $\Cal(\phi_{H_k}^{1})=1$ that $C^0$-converge to the identity by taking the supports of the $H_k$ to be small (see also Example \ref{grid} below).  This need for something like $C^0$ control on the Hamiltonian functions is consistent with other aspects of the theory of $C^0$ Hamiltonian dynamics as in \cite{OM}.

In \cite{OM}, the authors introduce the group $Hameo(M,d\lambda)$ of ``Hamiltonian homeomorphisms'' of the symplectic manifold $(M,d\lambda)$, and prove that it is a normal subgroup of the group $Sympeo^c(M,d\lambda)$ of compactly supported symplectic homeomorphisms of $(M,d\lambda)$, \emph{i.e.} of the $C^0$-closure  of the group of compactly supported symplectic diffeomorphisms in the group of compactly supported homeomorphisms of $M$.  As follows from a straightforward generalization of the discussion in \cite[Section 7]{Oh10}, an affirmative answer to Question \ref{mainq} would imply that $Hameo(M,d\lambda)$ is  a proper subgroup of $Sympeo^c(M,d\lambda)$, and hence that the latter group is not simple.  This is particularly interesting for $(M,d\lambda)$ equal to the open two-dimensional disk $D^2$ with its standard symplectic form, in which case $Sympeo^c(M,d\lambda)$ coincides with the compactly supported area-preserving homeomorphism group of $D^2$, whose simplicity or non-simplicity is a longstanding open problem.  In this case Question \ref{mainq} is equivalent to a slightly stronger version of \cite[Conjecture 6.8]{Oh10}, stating that $\Cal$ extends continuously to a homomorphism $Hameo(D^2,dx\wedge dy)\to\R$.   (To obtain Oh's conjecture precisely one would strengthen the hypothesis of Question \ref{mainq} to the statement that the isotopies $\{\phi_{k}^{t}\}_{t\in[0,1]}$ converge uniformly as $k\to\infty$ to some loop of homeomorphisms, instead of just assuming that their time-one maps converge to the identity.)

The main purpose of this note is to generalize a result of \cite{Oh16} related to Question \ref{mainq}, which we will recall presently after setting up Oh's notation and terminology.  Given our exact symplectic manifold
$(M,\lambda)$, let us write $\Lambda=\lambda\oplus(-\lambda)\in \Omega^1(M\times M)$, so that $(M\times M,\Lambda)$ is likewise an exact symplectic manifold.  By the Darboux-Weinstein Theorem, there is a symplectomorphism $\Psi\co \mathcal{U}_{\Delta}\to\mathcal{V}$ where $\mathcal{U}_{\Delta}$ is a neighborhood of the diagonal $\Delta\subset M\times M$ and $\mathcal{V}$ is a neighborhood of the zero-section in $T^{*}\Delta$ which is equipped with the symplectic form $-d\theta_{can}$ where $\theta_{can}\in \Omega^1(T^*\Delta)$ is the canonical one-form, with $\Psi|_{\Delta}$ equal to the inclusion of $\Delta$ as the zero-section. We fix such a Darboux-Weinstein chart $\Psi$. If $\phi\co M\to M$ is a compactly supported symplectic diffeomorphism, then its graph $\Gamma_{\phi}=\{(\phi(x),x)|x\in M\}$ is a Lagrangian submanifold of $(M\times M,d\Lambda)$, and if $\phi$ is sufficiently $C^0$-close to $1_M$ then $\Gamma_{\phi}$ will lie in the domain $\mathcal{U}_{\Delta}$ of $\Psi$.  One says that $\phi$ is \textbf{$\Psi$-graphical} if additionally the Lagrangian submanifold $\Psi(\Gamma_{\phi})\subset T^*\Delta$ coincides with the image of a section of the cotangent bundle $T^*\Delta\to\Delta$.  In particular this would hold if $\phi$ were assumed to be $C^1$-close to the identity; on the other hand it is possible for a diffeomorphism to be $\Psi$-graphical while still being fairly far away from the identity in the $C^1$-sense.  

One formulation of the main result of \cite{Oh16} is the following:
\begin{theorem}\cite[Theorem 1.10]{Oh16}\label{ohthm} For $M=D^2$, let $H_k\co [0,1]\times D^2\to\R$ be a sequence as in Question \ref{mainq}, and assume moreover that all of the diffeomorphisms $\phi_{H_k}^{1}$ are $\Psi$-graphical.  Then $\Cal(\phi_{H_k}^{1})\to 0$.
\end{theorem}

Thus, at least for $M=D^2$ the answer to Question \ref{mainq} is affirmative under an additional hypothesis; one might then hope to answer Question \ref{mainq} by finding a way to drop this hypothesis.  However we will prove the following generalization of Theorem \ref{ohthm}, which we interpret as suggesting that such a strategy raises more questions than might have been anticipated.

\begin{theorem}\label{mainthm}
Let $(M,\lambda)$ be an exact symplectic manifold, and let $\{\phi_k\}_{k=1}^{\infty}$ be a sequence in $Ham^c(M,d\lambda)$ with all $\phi_k$ generated by Hamiltonians that are supported in a fixed compact subset, such that $\xymatrix{\phi_k\ar[r]^{C^0}&1_M}$.  Assume moreover that each $\phi_k$ is $\Psi$-graphical.  Then $\Cal(\phi_k)\to 0$.
\end{theorem}

We emphasize that no assumption is made on the convergence of the Hamiltonians $H_k$ generating the $\phi_k$.  On the other hand, if one drops the graphicality hypothesis then an assumption similar to that in Question \ref{mainq} is certainly needed, see Example \ref{grid}.  This author believes that a major impediment to answering Question \ref{mainq} affirmatively is the current lack of a precise idea of what role the convergence of the Hamiltonian functions might play in the convergence of the Calabi invariants; Theorem \ref{mainthm} shows that this role only becomes essential when the $\phi_k$ are no longer graphical.

\begin{ex}\label{grid} On an arbitrary $2n$-dimensional exact symplectic manifold $(M,\lambda)$, for a sufficiently small $\delta>0$ we may consider a symplectically embedded copy of the cube $C^{2n}(\delta)=[0,\delta]^{2n}\subset M$ (with $C^{2n}(\delta)$ carrying the standard symplectic form $\sum_idx_i\wedge dy_i$).  Define a sequence of smooth functions $F_k\co M\to\R$ as follows.  Divide 
$C^{2n}(\delta)$ into $k^{2n}$ equal subcubes $C_{\vec{i}}^{(k)}$ for $\vec{i}=(i_1,\ldots,i_{2n})\in \{1,\ldots,k\}^{2n}$, by taking \[ C_{\vec{i}}^{(k)}=\prod_{j=1}^{2n}\left[\frac{i_j-1}{k}\delta,\frac{i_j}{k}\delta\right],\] and take $F_k$ to be a smooth function which is supported in the union of the interiors of the $C_{\vec{i}}^{(k)}$, obeys $0\leq F_k\leq 1$ everywhere, and which, for each $\vec{i}$, is equal to $1$ on a subset of $C_{\vec{i}}^{(k)}$ having measure at least $(1-1/k)(\delta/k)^{2n}$.  Define $H_k\co [0,1]\times M\to\R$ by $H_k(t,x)=F_k(x)$.

Evidently the sequence $H_k$ converges to the indicator function of $[0,1]\times C^{2n}(\delta)$ in every $L^p$ norm for $p<\infty$.  In particular $\Cal(\phi_{H_k}^{1})=\int_{M}H_k dt\wedge(d\lambda)^n\to \delta^{2n}>0$ as $k\to\infty$.  Meanwhile each $\phi_{H_k}^{1}$ acts as the identity on $M\setminus C^{2n}(\delta)$ and maps each of the subcubes $C_{\vec{i}}^{(k)}$ to themselves; since these subcubes have diameter $\sqrt{2n}\delta/k$ it follows that $\xymatrix{\phi_{H_k}^{1}\ar[r]^{C^0}&1_M}$. 
\end{ex}

Thus, if the answer to Question \ref{mainq} is to be affirmative, then its assumptions must be rather sharp: it is not sufficient for the Hamiltonian functions to converge in $L^p$ for any finite $p$, or to be uniformly bounded---one would need uniform convergence in the space variable.  Given the formula (\ref{calH}) for the Calabi homomorphism, one might naively have expected that $L^1$ convergence would be sufficient, but this is not the case.  

\begin{remark}
The uniqueness theorem for the Hamiltonians that generate Hamiltonian homeomorphisms \cite{V},\cite{BS} is somewhat reminiscent of Question \ref{mainq}.  Indeed this theorem (when $M$ is noncompact as it is in our case) can be phrased as stating that if $H_k\co [0,1]\times M\to \R$ are compactly supported smooth functions such that $\xymatrix{H_k\ar[r]^{L^{(1,\infty)}}&H}$ and if $\xymatrix{\phi_{H_k}^{t}\ar[r]^{C^0}& 1_M}$ uniformly in $t$, then $H\equiv 0$.  Thus, in comparison to this uniqueness theorem, Question \ref{mainq} asks whether one can obtain the weaker conclusion that $\int_MHdt\wedge(d\lambda)^n=0$ from the weaker hypothesis that only the time-one map $\xymatrix{\phi_{H_k}^{1}\ar[r]^{C^0}&1_M}$.  Note that Example \ref{grid} shows that the uniqueness theorem would fail to hold if we instead were to only assume that the functions $H_k$ converge in $L^p$ for some finite $p$, or that the $H_k$ are uniformly bounded, since in Example \ref{grid} we clearly have $\xymatrix{\phi_{H_k}^{t}\ar[r]^{C^0}&1_M}$ uniformly in $t$.  Thus the (potential) sharpness of the hypotheses in Question \ref{mainq} has some precedent in prior results. 
\end{remark}

\begin{remark}
En route to proving Theorem \ref{mainthm} we will prove a related result, Corollary \ref{general}, that does not make any graphicality assumptions and suggests a general viewpoint on Question \ref{mainq}.  Namely, given that $\xymatrix{\phi_{H_k}^{1}\ar[r]^{C^0}&1_M}$, Corollary \ref{general} shows that the statement that $\Cal(\phi_{H_k}^{1})\to 0$ is equivalent to the statement that the integrals of suitable pullbacks of \emph{generalized phase functions} $S_k$
 for the Lagrangian submanifolds $\Psi(\Gamma_{\phi_{H_k}^{1}})$ converge to zero.  
Here (as in \cite{BW}) a generalized phase function for a Lagrangian submanifold $L\subset T^*\Delta$ is a compactly supported smooth function $S_L\co L\to \R$ with $dS_L=\theta_{can}|_L$.  One way of constructing a generalized phase function for $L$ is to begin  with a generating function $S\co M\times\R^N\to \R$ (as in \cite{Sik}) with fiber critical set $\Sigma_S\subset M\times \R^N$ and canonical embedding $\iota_L\co\Sigma_S\to L$, and then define $S_L=S\circ \iota_{L}^{-1}$.  Thus in view of Corollary \ref{general}, Question \ref{mainq} leads to a question about the relationship between the behavior of a Hamiltonian function $H$ and that of the generating function of the graph of the time-one map of $H$.
\end{remark}

The following section contains the proofs of Corollary \ref{general} and Theorem \ref{mainthm}

\section{The Calabi invariant and generalized phase functions}

 As in the introduction, we work in a fixed exact symplectic manifold $(M,\lambda)$, and we fix a Darboux-Weinstein chart $\Psi\co \mathcal{U}_{\Delta}\to\mathcal{V}\subset T^*\Delta$ where $\Delta\subset (M\times M,d\Lambda)$ is the diagonal (and $\Lambda=\pi_{1}^{*}\lambda-\pi_{2}^{*}\lambda\in\Omega^1(M\times M)$ where $\pi_1,\pi_2\co M\times M\to M$ are the projections to the two factors).  The graph of $\phi$ is $\Gamma_{\phi}=\{(\phi(x),x)|x\in M\}\subset M\times M$.  (Throughout the paper our sign and ordering conventions are chosen to be consistent with \cite{Oh16}.) Assuming that $\Gamma_{\phi}\subset \mathcal{U}_{\Delta}$ we let \[ L_{\phi}=\Psi(\Gamma_{\phi})\subset T^*\Delta.\] Thus $\phi$ is graphical in the sense of the introduction if and only if there is $\alpha\in \Omega^1(\Delta)$ such that $L_{\phi}=\{(x,\alpha_x)|x\in \Delta\}$.  Such an $\alpha$ is necessarily closed, and we will argue below that it is exact. (This is not completely obvious since, if $\phi=\phi_{H}^{1}$, we allow the possibility that some $\Gamma_{\phi_{H}^{t}}$ is not contained in $\mathcal{U}_{\Delta}$.)
 
The following is well-known.
\begin{lemma}\label{f} Let $\phi\in Ham^c(M,d\lambda)$.  Then there is $f_{\lambda,\phi}\in C^{\infty}_{c}(M)$ such that $df_{\lambda,\phi}=\phi^{*}\lambda-\lambda$.  Moreover \begin{equation}\label{calf} \Cal(\phi)=\frac{1}{n+1}\int_{M}f_{\lambda,\phi}(d\lambda)^n.\end{equation}
\end{lemma}
\begin{proof} Assume that $\phi=\phi_{H}^{1}$ for some $H\in C^{\infty}_{c}([0,1]\times M)$.  Then \begin{align*} \phi^*\lambda-\lambda&=\int_{0}^{1}\frac{d}{dt}\phi_{H}^{t*}\lambda dt=\int_{0}^{1}\phi_{H}^{t*}\left(d\iota_{X_{H_t}}\lambda+\iota_{X_{H_t}}d\lambda\right)dt
\\ &= d\left(\int_{0}^{1}\left(\iota_{X_{H_t}}\lambda+H(t,\cdot)\right)\circ\phi_{H}^{t}dt\right) \end{align*} so we may take $f_{\lambda,\phi}=
\int_{0}^{1}\left(\iota_{X_{H_t}}\lambda+H(t,\cdot)\right)\circ\phi_{H}^{t}dt$.    The formula (\ref{calf}) is then a standard calculation involving repeated applications of Stokes' theorem, see \emph{e.g.} \cite[Proposition II.4.3]{B}.
\end{proof}

\begin{lemma}\label{R}
Assume that the domain of the Darboux-Weinstein chart $\Psi\co \mathcal{U}_{\Delta}\to\mathcal{V}$ has $\Delta$ as a deformation retract.  Then there is a smooth function $R\co \mathcal{V}\to \R$ such that $R|_{\Delta}\equiv 0$ and $-\theta_{can}=(\Psi^{-1})^{*}\Lambda+dR$.
\end{lemma}

\begin{proof} Since $\Psi$ is a symplectomorphism $(\mathcal{U}_{\Delta},d\Lambda)\to(\mathcal{V},-d\theta_{can})$, we have $d(\Psi^{*}\theta_{can}+\Lambda)=0$.  Of course, since $\Lambda|_{\Delta}=0$ while $\theta_{can}$ vanishes on the zero-section of $T^*\Delta$, and since $\Psi$ maps $\Delta$ to the zero-section, we have $(\Psi^*\theta_{can}+\Lambda)|_{\Delta}=0$.  So $\Psi^{*}\theta_{can}+\Lambda$ represents a class in the relative de Rham cohomology $H^1(\mathcal{U}_{\Delta},\Delta)$, which is trivial by the assumption that $\Delta$ is a deformation retract of 
$\mathcal{U}_{\Delta}$.  So there is $g\in C^{\infty}(\mathcal{U}_{\Delta})$ with $g|_{\Delta}=0$ such that $\Psi^{*}\theta_{can}+\Lambda=dg$.  So the lemma holds with $R=-g\circ\Psi^{-1}$.
\end{proof}

Putting together the two preceding lemmas gives the following.

\begin{prop}\label{sphi} Assume that $\phi\in Ham^c(M,d\lambda)$ has $\Gamma_{\phi}\subset\mathcal{U}_{\Delta}$ where $\Psi\co \mathcal{U}_{\Delta}\to\mathcal{V}$ is a Darboux-Weinstein chart whose domain $\mathcal{U}_{\Delta}$ has $\Delta$ as a deformation retract.  Let $f_{\lambda,\phi}$
be as in Lemma \ref{f}, and let $R$ be as in Lemma \ref{R}.  Where $L_{\phi}=\Psi(\Gamma_{\phi})$, define $S_{\phi}\co L_{\phi}\to\R$ by \begin{equation}\label{Sformula} S_{\phi}=R+f_{\lambda,\phi}\circ\pi_2\circ\Psi^{-1}.\end{equation}  Then $S_{\phi}\co L_{\phi}\to M$ is a compactly supported smooth function satisfying $dS_{\phi}=-\theta_{can}|_{L_{\phi}}$.
\end{prop}

\begin{proof} That $S_{\phi}$ is compactly supported follows from the facts that $L_{\phi}$ coincides outside of a compact subset with $\Delta$, on which $R$ vanishes identically, that $\pi_2\circ\Psi^{-1}$ maps $L_{\phi}$ diffeomorphically to $M$, and that $f_{\lambda,\phi}$ is compactly supported. 

By the defining property of $R$ we have \begin{equation}\label{Rtheta} (dR+\theta_{can})|_{L_{\phi}}=(-\Psi^{-1*}\Lambda)|_{L_{\phi}}=-\Psi^{-1*}(\Lambda|_{\Gamma_{\phi}}).\end{equation}  Meanwhile since (by the definition of $\Gamma_{\phi}$ as $\{(\phi(x),x)|x\in M\}$) we have $\phi\circ\pi_2|_{\Gamma_{\phi}}=\pi_1|_{\Gamma_{\phi}}$, we see that \begin{equation}\label{df} d\left(f_{\lambda,\phi}\circ\pi_2|_{\Gamma_{\phi}}\right)=(\pi_2|_{\Gamma_{\phi}})^{*}(\phi^*\lambda-\lambda)=(\pi_{1}^{*}\lambda-\pi_{2}^{*}\lambda)|_{\Gamma_{\phi}}=\Lambda|_{\Gamma_{\phi}}.\end{equation}  Combining (\ref{Rtheta}) and (\ref{df}) gives \[ \Psi^{*}\left((dR+\theta_{can})|_{L_{\phi}}\right)=-d\left(f_{\lambda,\phi}\circ\pi_2|_{\Psi^{-1}(L_{\phi})}\right),\] from which the result follows immediately.
\end{proof}

The key point now is that the conditions that $S_{\phi}$ be compactly supported and that $dS_{\phi}=-\theta_{can}|_{L_{\phi}}$ uniquely determine the smooth function $S_{\phi}\co L_{\phi}\to \R$.  A function satisfying these properties is (perhaps after a sign reversal) sometimes called a ``generalized phase function;'' the formula (\ref{Sformula}) for such a function together with the formula (\ref{calf}) thus relate generalized phase functions to the Calabi homomorphism.  The relation is especially simple in the $C^0$-small, graphical case, but first we note a consequence that does not require graphicality.

\begin{cor}\label{general} Assume that $\{\phi_k\}_{k=1}^{\infty}$ is a sequence in $Ham(M,d\lambda)$ with each $\phi_k$ supported in a fixed compact set, such that $\xymatrix{\phi_k\ar[r]^{C^0}& 1_M}$.  Construct the Lagrangian submanifolds $L_{\phi_k}=\Psi(\Gamma_{\phi_k})\subset T^*\Delta$ (for sufficiently large $k$) as above, and suppose that $S_k\co L_{\phi_k}\to \R$ are compactly supported smooth functions obeying $dS_k=-\theta_{can}|_{L_{\phi_k}}$.  Then $\Cal(\phi_k)\to 0$ \textbf{if and only if} \begin{equation}\label{Sint} \int_{M}(S_k\circ\Psi\circ(\phi_k\times 1_M))(d\lambda)^n\to 0.\end{equation}
\end{cor}

\begin{proof} By shrinking the domain of $\Psi\co \mathcal{U}_{\Delta}\to\mathcal{V}$ and perhaps removing an initial segment of the sequence $\{\phi_k\}_{k=1}^{\infty}$ we may assume that $\mathcal{U}_{\Delta}$ has $\Delta$ as a deformation retract and that each $\Gamma_{\phi_k}\subset\mathcal{U}_{\Delta}$ so that we are in the setting of Proposition \ref{sphi}.  

Note that $\pi_2\circ\Psi^{-1}\co L_{\phi_{k}}\to M$ is a diffeomorphism with inverse $\Psi\circ(\phi_k\times 1_M)$.  
So (\ref{Sformula}) gives \begin{equation}\label{fsr} f_{\lambda,\phi_k}=(S_k-R)\circ\Psi\circ(\phi_k\times 1_M).\end{equation}  But $R$ is a smooth function, independent of $k$, which vanishes along the zero section of $T^*\Delta$, so $R\circ\Psi$ vanishes along the diagonal $\Delta\subset M\times M$.  So the assumption that $\xymatrix{\phi_k\ar[r]^{C^0}&1_M}$ implies that $\max|R\circ\Psi\circ(\phi_k\times 1_M)|\to 0$ as $k\to\infty$, and the assumption there is a compact set simultaneously containing the supports of all of the $\phi_k$ implies that the support of each $R\circ\Psi\circ(\phi_k\times 1_M)$   is contained in this same compact set.  Hence $\int_M R\circ\Psi\circ(\phi_k\times 1_M)(d\lambda)^n\to 0$, in view of which the corollary follows directly from (\ref{fsr}) and (\ref{calf}).
\end{proof}

\begin{proof}[Proof of Theorem \ref{mainthm}] 
We work in the same setting as Corollary \ref{general}, with the additional assumption that each $\phi_k$ is $\Psi$-graphical.  As noted at the start of the proof of Corollary \ref{general}, we may assume that $\mathcal{U}_{\Delta}$ deformation retracts to $\Delta$.  

Since $\phi_k$ is $\Psi$-graphical, let $\alpha_k\in \Omega^1(\Delta)$ have the property that $L_{\phi_k}$ is the image of $\alpha_k$ (when the latter is viewed as a map $\Delta\to T^*\Delta$).  Recall that, by the definition of the canonical one-form $\theta_{can}$, $\alpha_{k}^{*}\theta_{can}=\alpha_k$.  Denoting by $\pi_{\Delta}\co T^*\Delta\to\Delta$ the bundle projection, $\pi_{\Delta}|_{L_{\phi_k}}\co L_{\phi_k}\to \Delta$ is a diffeomorphism with inverse $\alpha_k$, so it follows that $\theta_{can}|_{L_{\phi_k}}=(\pi_{\Delta}|_{L_{\phi_k}})^{*}\alpha_k$.  So Proposition \ref{sphi} shows that the function $S_{\phi_k}\co L_{\phi_k}\to\R$ defined therein obeys $dS_{\phi_k}=-\theta_{can}|_{L_{\phi_k}}=-(\pi_{\Delta}|_{L_{\phi_k}})^{*}\alpha_k$ and hence \begin{equation}\label{sak} \alpha_k=-d(S_{\phi_k}\circ\alpha_k).\end{equation}
In particular $\alpha_k$ is the derivative of a compactly supported smooth function.

By assumption each $\phi_k$ is generated by a Hamiltonian $H_k$ with support contained in some fixed compact set $[0,1]\times K$ where $K\subset M$.  As is clear from the formula for $f_{\lambda,\phi_k}$ in the proof of Lemma \ref{f}, the functions $f_{\lambda,\phi_k}$ likewise have support contained in $K$.  So since $L_{\phi_k}$ coincides with the zero-section outside of $\Psi(\mathcal{U}_{\Delta}\cap (K\times K))$,  we see from (\ref{Sformula}) that $S_{\phi_k}$ vanishes at all points of form $\Psi(x,x)$ for $x\notin K$. So (identifying $K$ with its image in $\Delta$ by the diagonal embedding) the functions $-S_{\phi_k}\circ\alpha_k$ are all likewise supported in $K$.   

Now let us fix a Riemannian metric on $\Delta$ and denote by $A$ the diameter of $K$ with respect to $g$.  Also let $x_0\in \Delta\setminus K$ be of distance less than $1$ from $K$, so for any $x\in \Delta$ there is a path $\gamma_x\co[0,1]\to \Delta$ having $|\gamma'_{x}(t)|_g\leq A+1$ for all $t$.  Since $x_0\in\Delta\setminus K$, we have $S_{\phi_k}\circ\alpha_{k}(x_0)=0$.  Thus, for any $x\in K$, using (\ref{sak}) we have \begin{align*} S_{\phi_k}\circ\alpha_k(x) &=\int_{0}^{1}\frac{d}{dt}\left(S_{\phi_k}\circ\alpha_k(\gamma_{x}(t))\right)dt\leq (A+1)\max_{t\in [0,1]}|d(S_{\phi_k}\circ\alpha_k)(\gamma_x(t))|_g
\\&\leq (A+1)\max_{\Delta}|\alpha_k|_g.\end{align*}  Since for $x\notin K$ we have $S_k\circ\alpha_k(x)=0$, and since $\alpha_k\co \Delta\to L_{\phi_k}$ is surjective, this proves that \[ \max_{L_{\phi}}|S_{\phi_k}|\leq (A+1)\max_{\Delta}|\alpha_k|_g.\]  

But now the result is immediate from Corollary \ref{general}: the integrand of (\ref{Sint}) vanishes outside of $K$, so the integral is bounded above by \[ (A+1)\max_{\Delta}|\alpha_k|_{g}\int_{K}(d\lambda)^n.\]  The assumption that $\xymatrix{\phi_k\ar[r]^{C^0}&1_M}$ shows that $\max_{\Delta}|\alpha_k|_g\to 0$, so $\Cal(\phi_k)\to 0$ as desired.
\end{proof}

\subsection*{Acknowledgements} This work was motivated by many conversations with Y.-G. Oh related to Question \ref{mainq}, some of which occurred during a visit to the IBS Center for Geometry and Physics in May 2015; I thank IBS for its hospitality.  The work was partially supported by the NSF through grant DMS-1509213.

\end{document}